\documentclass[12pt]{article}
\usepackage{amsmath, amsthm, amssymb}

\title{Estimate of the number of one-parameter
families
of modules over a tame algebra\footnotetext{This is a preliminary version of the paper published in Linear Algebra Appl. 365 (2003) 115--133.}}

\author{Thomas Br\"ustle\\
 Fakult\"at f\"ur Mathematik,
 Universit\"at Bielefeld\\
 Postfach 100 131 D-33501 Bielefeld,
 Germany\\
 bruestle@mathematik.uni-bielefeld.de
\medskip\\ and\medskip\\
Vladimir V. Sergeichuk%
\thanks{The research was done while this
author was visiting the University of
Bielefeld and the University of Utah
supported by Sonderforschungsbereich 343 and
NSF grant DMS-0070503.}\\ Institute of
Mathematics\\ Tereshchenkivska 3, Kiev,
Ukraine\\sergeich@imath.kiev.ua}
\date{}

\begin{document}

\newcommand{\im}{\mathop{\rm Im}\nolimits}
\newcommand{\diag}{\mathop{\rm diag}\nolimits}
\newcommand{\Ker}{\mathop{\rm Ker}\nolimits}
\newcommand{\rad}{\mathop{\rm Rad}\nolimits}
\newcommand{\Aut}{\mathop{\rm Aut}\nolimits}
\newcommand{\End}{\mathop{\rm End}\nolimits}

\renewcommand{\le}{\leqslant}
\renewcommand{\ge}{\geqslant}

\newtheorem{theorem}{Theorem}[section]
\newtheorem{lemma}{Lemma}[section]

\maketitle

\begin{abstract}
The problem of classifying modules over a
tame algebra $A$ reduces to a block matrix
problem of tame type whose indecomposable
canonical matrices are zero- or
one-parameter. Respectively, the set of
nonisomorphic indecomposable modules of
dimension at most $d$ divides into a finite
number $f(d,A)$ of modules and one-parameter
series of modules.

We prove that the number of canonical
parametric block matrices of size $m\times
n$ and a given partition into blocks is
bounded by $4^s$, where $s$ is the number of
free entries, $s\le mn$. Basing on this
estimate, we prove that
$$
f(d,A)\le {\binom {d+r} r}
4^{d^2(\delta_1^2+\dots +\delta_r^2)}\le
(d+1)^r4^{d^2(\dim A)^2},
$$
where $r$ is the number of nonisomorphic
indecomposable projective left $A$-modules
and $\delta_1,\dots,\delta_r$ are their
dimensions.

{\it AMS classification:} 15A21; 16G60.

{\it Keywords:} Canonical matrices;
Classification; Tame algebras.
 \end{abstract}

\section{Introduction}   \label{s1}

Matrices and finite dimensional algebras are
considered over an algebraically closed
field $k$.

Gabriel, Nazarova, Roiter, Sergeichuk, and
Vossieck \cite{gab_vos} studied matrix
problems, in which the row-transformations
are given by a category and the column
transformations are arbitrary. They
interpreted ${m\times n}$ matrices as points
of the affine space $k^{m\times n}$ of all
${m\times n}$ matrices and proved that for a
tame matrix problem and every ${m\times n}$
there exists a full system of nonisomorphic
indecomposable $m\times n$ matrices that
consists of a finite number of points and
punched straight lines. This result was
extended to modules over a tame finite
dimensional algebra $A$: for every
$d\in\mathbb{N}$ there exists an almost full
(except for a finite number of modules)
system of nonisomorphic indecomposable
$d$-dimensional modules that consists of a
finite number $\rho_A(d)$ of punched lines
(an $A$-module of dimension $d$ was
considered as a point of the affine space
$k^{d\times d}\oplus\dots \oplus k^{d\times
d}$; the number of summands $k^{d\times d}$
is a number of generators of $A$).

Br\"ustle \cite{bru} proved, that
\begin{equation}\label{0.1}
\rho_A(d)\le \dim (\rad A)\cdot e^{2^6
3^{d-1} (d-1)^{2d-1}}.
\end{equation}

Sergeichuk \cite{ser} extended the results
of \cite{gab_vos} to block matrix problems
in which rows and columns transformations
are given by triangular matrix algebras: If
the matrix problem is of tame type, then for
every $m\times n$ there exists a finite set
of zero- and one-parameter matrices
\begin{equation}\label{0.2}
  M_1,\dots,M_{t_1},\, N_1(\lambda_1),
  \dots, N_{t_2}(\lambda_{t_2})
\end{equation}
such that the set of indecomposable
canonical $m\times n$ matrices is $$
\{M_1,\dots,M_{t_1}\}\cup \{N_1(a)\,|\, a\in
k\}\cup \dots \cup \{N_{t_2}(a)\,|\, a\in
k\}; $$ it may be interpreted as a set of
points and straight lines in the affine
space $k^{m\times n}$. The proof was based
on Belitski\u\i's algorithm \cite{bel} (see
also \cite{bel1}) for reducing a matrix to
canonical form; two matrices may be reduced
one to the other if and only if they have
the same canonical form.

Drozd \cite{dro1} proposed the following
reduction of the problem of classifying
modules over an algebra $A$ to a matrix
problem. Let $P_1,\dots, P_r$ be all
nonisomorphic indecomposable projective
right $A$-modules. For every right module
$M$ over $A$, there exists an exact sequence
\begin{equation*}\label{0.3}
P_1^{p_1}\oplus\dots\oplus P_r^{p_r}
\stackrel{\varphi}{\longrightarrow}
P_1^{q_1}\oplus\dots\oplus P_r^{q_r}
\stackrel{\psi}{\longrightarrow}
M\longrightarrow 0,
\end{equation*}
where $X^{l}:=X\oplus\dots\oplus X$ ($l$
times). The homomorphism $\varphi$ is
determined up to transformations
$\varphi\mapsto g\varphi f$, where $f$ and
$g$ are automorphisms of $\oplus_iP_i^{p_i}$
and $\oplus_iP_i^{q_i}$. The $\varphi$, $f$,
and $g$ can be given by their matrices in
bases of the spaces $\oplus_iP_i^{p_i}$ and
$\oplus_iP_i^{q_i}$ over $k$. This reduces
the problem of classifying modules over
algebras to block matrix problems, which
were studied in \cite{ser}. The modules that
correspond to the canonical matrices form a
full system of nonisomorphic modules;
indecomposable modules correspond to
indecomposable matrices.

In this article, we obtain the following
estimates:
\begin{itemize}
  \item [(i)] If a block matrix problem is of
  tame type, then the number of canonical
parametric block matrices \eqref{0.2} of
size $m\times n$ and a given partition into
blocks is bounded by $4^s$, where $s$ is the
number of free entries, $s\le mn$.

  \item [(ii)] If an algebra $A$ is of
  tame type, then the number of zero-
  and one-parameter matrices that give
  a full system of nonisomorphic
  indecomposable modules of dimension
  at most $d$ is bounded by
 $$
{\binom {d+r} r} 4^{d^2(\delta_1^2+\dots
+\delta_r^2)},
$$
where $r$ is the number of nonisomorphic
indecomposable projective left $A$-modules
and $\delta_1,\dots,\delta_r$ are their
dimensions.
\end{itemize}
Here the first estimate is optimal and the second one improves significantly the estimate
from [3].
The paper is organized as follows: in Section 2, we introduce the concept of standard
linear matrix problems and recall Belitskii's algorithm. Section 3 is devoted to
the proof of the estimate (i), Section 4 is concerned with the corresponding estimate
(ii) for modules over a tame algebra.

\section{Belitski\u{\i}'s algorithm
for linear matrix problems} \label{s2}

A block matrix $M=[M_{ij}]$, $M_{ij}\in
k^{m_i\times n_j}$, will be called an
$\underline{m}\times\underline{n}$ {\it
matrix}, where
$\underline{m}=(m_1,m_2,\ldots)$ and
$\underline{n}=(n_1,n_2,\ldots)$.

A linear matrix problem is the canonical
form problem for
$\underline{n}\times\underline{n}$ matrices
whose blocks satisfy a certain system of
linear homogeneous equations. Solving this
system, we select {\it free blocks} that are
arbitrary; the other blocks are their linear
combinations. The set of admissible
transformations consists of elementary
transformations within strips, additions of
linear combinations of rows of the $i$th
strip to rows of the $j$th strip for certain
$i>j$, and additions of linear combinations
of columns of the $i$th strip to columns of
the $j$th strip for certain $i<j$.
Elementary transformations and additions may
be linked: making elementary transformations
within a horizontal strip, we must produce
the same elementary transformations within
all horizontal strips linked with it and
inverse elementary transformations within
all vertical strips linked with it. Making
an addition between strips, we must produce
all linked with it additions.

Applying Belitski\u\i's algorithm
(\cite{bel},\cite{ser}), we can reduce a
block matrix by these transformations to
canonical form; two block matrices may be
reduced one to the other if and only if they
have the same canonical form.

If the matrix problem is of tame type (that
is, it does not contain the problem of
classifying pairs of matrices up to
simultaneous similarity, then the set of
direct-sum-indecomposable canonical
$\underline{n}\times\underline{n}$ matrices
forms a finite number of points and straight
lines in the affine space of
$\underline{n}\times\underline{n}$ matrices
(see \cite[Theorem 3]{ser}). In the article,
we prove that this number is bounded by
$4^s$, where $s$ is the number of entries in
free blocks.

Let us sketch a more formal definition of a
linear matrix problem (see \cite[Sect.
2.2]{ser}).

An algebra $\varGamma \subset k^{t\times t}$
of upper triangular matrices is a {\it basic
matrix algebra} if
\begin{equation}\label{3.00}
\begin{bmatrix}
             a_{11}&\cdots&a_{1t} \\
                 &\ddots&\vdots \\
              \text{\Large 0} & & a_{tt}
            \end{bmatrix}\in\varGamma
\quad  {\rm implies} \quad
        \begin{bmatrix}
           a_{11}  & & \text{\Large 0} \\
                       &\ddots&   \\
              \text{\Large 0}  & & a_{tt}
         \end{bmatrix}\in\varGamma.
\end{equation}
The diagonals $(a_{11},a_{22},\dots,a_{tt})$
of the matrices from $\varGamma$ form a
subspace in $k^t =k\oplus\dots\oplus k$,
which may be given by a system of equations
of the form $a_{ii}=a_{jj}$. Define an
equivalence relation in $T=\{1,\dots,t\}$
putting
\begin{equation}\label{4.0a}
\text{$i\sim j$ if and only if
$\diag(a_1,\dots,a_t)\in \varGamma$ implies
$a_i=a_j$.}
\end{equation}
We say that a sequence of nonnegative
integers $\underline{n}=(n_1,n_2,\dots,n_t)$
is a {\it step-sequence} if $i\sim j$
implies $n_i=n_j$.

A {\it linear matrix problem given by a
pair}
\begin{equation}     \label{3.4aa}
(\varGamma,\cal M), \quad \varGamma {\cal
M}\subset {\cal M},\ {\cal M}\varGamma
\subset {\cal M},
\end{equation}
consisting of a basic $t\times t$ algebra
$\varGamma$ and a vector space ${\cal
M}\subset k^{t\times t}$, is the canonical
form problem for matrices $M\in {\cal
M}_{\underline{n}\times\underline{n}}$ with
respect to transformations
\begin{equation} \label{3.01}
M\longmapsto S^{-1}MS,\qquad S\in
\varGamma_{\underline{n}\times
\underline{n}}^{*},
\end{equation}
where $\underline{n}=(n_1,\dots,n_t)$ is a
step-sequence,
$\varGamma_{\underline{n}\times
\underline{n}}$ and ${\cal
M}_{\underline{n}\times\underline{n}}$
consist of
$\underline{n}\times\underline{n}$ matrices
whose blocks satisfy the same systems of
linear homogeneous equations as the entries
of $t\times t$ matrices from $\varGamma$ and
$\cal M$, respectively, and
$\varGamma_{\underline{n}\times
\underline{n}}^{*}$ denotes the set of
nonsingular matrices from
$\varGamma_{\underline{n}\times
\underline{n}}$. ($\varGamma$ and ${\cal M}$
are subspaces of $k^{t\times t}$; they may
be given by systems of linear homogeneous
equations of the form
$$
 \sum_{(i,j)\in{\cal
I} \times{\cal J}}d_{ij}x_{ij} = 0,
$$
where ${\cal I},{\cal J}\in
\{1,\dots,t\}/\!\sim$ are equivalence
classes.)

Let us outline Belitski\u\i's algorithm (it
has been detailed in \cite{ser}) for
reducing a matrix
 $$
 M=\begin{bmatrix}
       M_{11} & \cdots  & M_{1t} \\
     \hdotsfor{3} \\
        M_{t1} & \cdots  & M_{tt}
    \end{bmatrix}
\in{\cal
M}_{\underline{n}\times\underline{n}} $$ to
canonical form by transformations
\eqref{3.01}. We assume that the blocks of
$M$ (and of every block matrix) are ordered
starting from the lower strip:
\begin{equation}      \label{4.1}
M_{t1}<M_{t2}<\dots<M_{tt}<M_{t-1,1}<
M_{t-1,2}<\dots<M_{t-1,t}<\cdots
\end{equation}
In the set $\{M_{ij}\}$ of blocks of $M$, we
select the set of free blocks such that
every unfree block is a linear combination
of free blocks that preceding it with
respect to the ordering \eqref{4.1}. The
entries of free blocks will be called the
{\it free entries}.

On the first step, we reduce the block
$M_{t1}$. It is reduced by transformations
\begin{equation} \label{5.1}
M_{t1}\longmapsto
S^{-1}_{tt}M_{t1}S_{11},\qquad S\in
\varGamma_{\underline{n}\times
\underline{n}}^*.
\end{equation}

If $1\nsim t$, then $M_{t1}$ is reduced by
arbitrary equivalence transformations. We
reduce it to the form
\begin{equation} \label{5.2}
      \left[ \begin{array}{cc}
                    0  &  I      \\
                    0  &  0
            \end{array} \right]
\end{equation}
and extend its division into substrips onto
the first vertical and the first horizontal
strips of $M$.

If $1\sim t$, then $M_{t1}$ is reduced by
arbitrary similarity transformations. We
reduce it to a {\it Weyr matrix} (which is
obtained from a Jordan matrix by
simultaneous permutations of rows and
columns, see \cite[Sect. 1.3]{ser}):
\begin{equation}       \label{5.3}
W=W_{\alpha_1}\oplus\dots\oplus
W_{\alpha_r},\quad \alpha_1\prec\dots\prec
\alpha_r,
\end{equation}
where $\prec$ is a linear order in $k$ (if
$k$ is the field of complex numbers, we use
the lexicographic ordering), and
\begin{equation}       \label{5.4}
W_{\alpha_i}= \left[\begin{tabular}{cccc}
$\alpha_iI_{m_{i1}}$&$W_{i1}$&&{\Large 0}\\
&$\alpha_iI_{m_{i2}}$&$\ddots$&\\
&&$\ddots$&$W_{i,q_i-1}$\\ {\Large
0}&&&$\alpha_iI_{m_{iq_i}}$
\end{tabular}\right],\quad
W_{ij}=
\begin{bmatrix}
I\\0
\end{bmatrix},
\end{equation}
$m_{i1}\ge\dots\ge m_{iq_i}$. We make the
most coarse partition of $W$ into substrips
for which all diagonal subblocks have the
form $\alpha_i I$ and all off-diagonal
subblocks are $0$ and $I$ (all matrices
commuting with $W$ are upper block
triangular with respect to this partition).
We extend this division of $M_{t1}=W$ into
substrips onto the first vertical and the
first horizontal strips of $M$.

Then we restrict the set of admissible
transformations with $M$ to those
transformations \eqref{5.1} that preserve
$M_{t1}$ (that is, $S^{-1}_{tt}M_{t1}S_{11}=
M_{t1}$). It may be proved that the algebra
of matrices $$ \Lambda_1= \{S=[S_{ij}]\in
\varGamma_{\underline{n}\times
\underline{n}}\, |\, M_{t1}S_{11}=
S_{tt}M_{t1}\} $$ also has the form
$\varGamma'_{\underline{n'}\times
\underline{n'}}$, where $\varGamma'$ is a
basic matrix algebra. The entries of
$M_{t1}$ are the {\it reduced entries} of
$M$.

On the second step, we take the first
unreduced (that is, does not contained in
$M_{t1}$) block with respect to the new
partition and reduce it.

On each step, we take the first unreduced
block $M_{pq}$ (with respect to a new
subdivision) and reduce it by those
admissible transformations that preserve all
reduced entries. If $M_{pq}$ is not free,
then it is the linear combination of
preceding free blocks that have been
reduced, and hence $M_{pq}$ is not changed
at this step. If $M_{pq}$ is free, then the
following three cases are possible:

(i) There exists a nonzero admissible
addition to $M_{pq}$ from other blocks.
Since admissible transformations are given
by upper block triangular matrices and we
use the ordering \eqref{4.1}, all nonzero
additions to $M_{pq}$ are from preceding
(reduced) blocks. We make $M_{pq}=0$ by
these additions.

(ii)  There exist no nonzero admissible
additions to $M_{pq}$ and it is reduced by
equivalence transformations. Then we reduce
$M_{pq}$ to the form \eqref{5.2}.

(iii)  There exist no nonzero admissible
additions to $M_{pq}$ and it is reduced by
similarity transformations. Then we reduce
$M_{pq}$ to a Weyr matrix.

At the end of this step, we make an
additional subdivision of $M$ into strips in
accordance with the block form of the
reduced $M_{pq}$ and restrict the set of
admissible transformations to those that
preserve $M_{pq}$.

The process stops after reducing the last
unreduced entry of $M$. The obtained
canonical matrix will be partitioned into
\begin{equation} \label{9}
M_1,M_2,\dots,M_{l(M)},
\end{equation}
where $M_i$ is the block that reduces at the
$i$th step. Each $M_i$ has the form $0$,
\eqref{5.1}, or is a Weyr matrix. We will
call \eqref{9} the {\it boxes} of $M$.

For instance,
\begin{equation*}      
M=  \left[  \begin{tabular}{c|c}
   $M_3$  &\!\!\!\!\!  \begin{tabular}{c|c}
       $M_6$ & $M_7$  \\  \hline
        $M_4$ &$ M_5$
 \end{tabular}\!\!\!
\\  \hline
                      $ M_1$  & $M_2$
            \end{tabular} \right]
=\left[ \begin{tabular}{c|c}
   $ \!\!\!\! \begin{array}{cc}-1&1\\
0&-1\end{array}$\!\!\!\! &
 \!\!\!\begin{tabular}{c|c}2&$0$ \\
 \hline 0&1\end{tabular}\!\!\! \\  \hline
                      $ 3I_2$  & $0$
            \end{tabular} \right],
            \qquad l(M)=7,
\end{equation*}
is a canonical $(2,2)\times(2,2)$ matrix for
the linear matrix problem given by the pair
$(\varGamma, k^{2\times 2})$, where
 $$ \varGamma=
\left\{\left.\begin{bmatrix} a&b\\ 0&a
\end{bmatrix}\,\right|\, a,b\in k\right\}.
$$

Let $M$ be a canonical matrix. Replacing all
diagonal entries of its free boxes that are
Weyr matrices by parameters, we obtain a
parametric matrix $M(\lambda_1,\dots,
\lambda_p)$. Its {\it domain of parameters}
$\cal{D}$ is the set of all
$(a_1,\dots,a_p)\in k^p$ for which
$M(a_1,\dots,a_p)$ is a canonical matrix. If
a parameter $\lambda_i$ is finite (that is,
the number of vectors of $\cal{D}$ with
distinct $a_i$ is finite), we replace
$\lambda_i$ by its values and obtain several
parametric matrices with a smaller number of
parameters. Repeating this process, we
obtain parametric matrices having only
infinite parameters. The obtained matrices
will be called {\it canonical parametric
matrices}.

Hence, the canonical form problem for
 $\underline{n}\times\underline{n}$
matrices with the same $\underline{n}$
reduces to the problem of finding a finite
number of canonical parametric matrices and
their domains of parameters.

\section{Estimate of the number of
canonical parametric matrices} \label{s3}

In this section, we study a linear matrix
problem of tame type. As was proved in
\cite{ser}, each of its canonical parametric
matrices, up to simultaneous permutations of
rows and columns, has the form
\begin{equation} \label{12}
N_1(\lambda_1)\oplus\dots\oplus
N_p(\lambda_p) \oplus R_1\oplus\dots\oplus
R_q,\qquad p\ge 0,\quad q\ge 0,
\end{equation}
where $N_i(\lambda_i)$ and $R_j$ are
indecomposable canonical one- and
zero-parameter canonical matrices. The
purpose of the section is to prove the
following theorem.

\begin{theorem}   \label{t1}
If a linear matrix problem is of tame type,
then the number of its canonical parametric
matrices of size
$\underline{n}\times\underline{n}$ is
bounded by $4^{s(\underline{n})}$, where
$s(\underline{n})$ is the number of free
entries in an
$\underline{n}\times\underline{n}$ matrix.
\end{theorem}

We first prove a technical lemma.

\begin{lemma} \label{l1}
Let
\begin{equation}\label{14.0}
A(x,y)=\begin{bmatrix}
  a_{11}(x,y) &\dots & a_{1n}(x,y) \\
  \hdotsfor{3} \\
  a_{m1}(x,y) & \dots & a_{mn}(x,y)
\end{bmatrix}
\end{equation}
be a matrix whose entries are linear
polynomials in $x$ and $y$, and let the rows
of $A(\alpha,\beta)$ be linearly independent
for all $(\alpha,\beta)\in k^2$ except for
$$ (\alpha_1,\beta_1),\
(\alpha_2,\beta_2),\dots,
(\alpha_s,\beta_s). $$ Then $s\le m^2$;
moreover, $s\le 3$ if $m=2$.
\end{lemma}

\begin{proof}
{\it Part 1: $s\le m^2$}. Clearly, $m\le n$.
The rows of $A(\alpha,\beta)$ are linearly
dependent if and only if $(\alpha,\beta)\in
k^2$ is a common root of all determinants
formed by columns of $A(x,y)$. The
determinants are polynomials in $x$ and $y$
of degree at most $m$; they are relatively
prime (otherwise, they have infinitely many
common roots $(\alpha,\beta)\in k^2$). The
inequality $s\le m^2$ follows from the
following statement:
\begin{equation}       \label{14.1}
\parbox{25em}
{If $h_1,\dots,h_t\in k[x,y]$ are
polynomials of degree at most $m$ and their
greatest common divisor $(h_1,\dots,h_t)$ is
1, then they have at most $m^2$ common
roots.}
\end{equation}
For $m=2$, this statement is a partial case
of the Bezout theorem \cite[Sect. 1.3]{gri}:
if $h_1,h_2\in k[x,y]$ and $(h_1,h_2)=1$,
then they have at most
$\deg(h_1)\cdot\deg(h_2)$ common roots.

Let $m\ge 3$. Applying induction in $t$, we
may assume that $d:=(h_1,\dots,h_{t-1})\ne
1$. If $(\alpha,\beta)$ is a common root of
$h_1,\dots,h_{t}$, then $(\alpha,\beta)$ is
a root of $h_t$ and also a root of $d$ or a
common root of
$g_1=h_1/d,\dots,g_{t-1}=h_{t-1}/d$. By the
Bezout theorem, the number of common roots
of $d$ and $h_t$ is at most $\deg(d)m$.  By
induction, the number of common roots of
$g_1,\dots,g_{t-1}$ is at most
$(m-\deg(d))^2$. Hence, the number of common
roots of $h_1,\dots,h_{t}$ is at most
$\deg(d)m+(m-\deg(d))^2\le \deg(d)m+
(m-\deg(d))m=m^2$. This proves \eqref{14.1}.
\medskip

{\it Part 2: $s\le 3$ if $m=2$}. Let $m=2$;
assume to the contrary that $s>3$. We will
reduce $A(x,y)$ by elementary
transformations over $k$ and by
substitutions $$
\begin{tabular}{l}
  $x_{\text{new}}=ax+by+c,$ \\
  $y_{\text{new}}=a_1x+b_1y+c_1,$
\end{tabular}\quad
\begin{vmatrix}
  a&b \\
  a_1&b_1
\end{vmatrix}\ne 0;
$$ the obtained matrices $A'(x,y)$ will have
the same number $s$, and their entries are
linear polynomials too. We suppose that each
of the matrices $A'(x,y)$ does not contain a
zero column; otherwise we can remove it and
take the obtained matrix instead of
$A(x,y)$.

Let $n=2$. The rows of $A(\alpha,\beta)$ are
linearly independent only if $\det
A(\alpha,\beta)\ne 0$. Under the conditions
of the lemma, the rows of $A(\alpha,\beta)$
are linearly independent for almost all
$(\alpha,\beta)\in k^2$, and so $\det
A(x,y)$ is a nonzero scalar and the rows of
$A(\alpha,\beta)$ are linearly independent
for all $(\alpha,\beta)\in k^2$.

Hence, $n\ge 3$. By elementary
transformations of rows of $A(x,y)$, we make
$a_{11}(x,y)=a_{11}\in \{0,1\}$.

If $a_{21}(x,y)=a_{21}\in k$, we make
$(a_{11},a_{21})=(1,0)$ by elementary
transformations of rows. The rows of
$A(\alpha,\beta)$ are linearly dependent
only if $$ a_{22}(\alpha,\beta)=
a_{23}(\alpha,\beta)= \dots =
a_{2n}(\alpha,\beta)=0. $$ Since
$a_{22}(x,y), a_{23}(x,y), \ldots$ are
linear polynomial, $s\le 1$.

Hence $a_{21}(x,y)\notin k$. We make
$a_{21}(x,y)=x$ by the substitution $$
 x_{\text{new}}= a_{21}(x,y),\quad
 y_{\text{new}}=
 \begin{cases}
 y & \text{if $a_{21}(x,y)\notin k[y]$},\\
 x & \text{otherwise}.
 \end{cases}
$$

If there exist distinct $l,r>1$ such that
\begin{equation}  \label{15.0}
\begin{tabular}{l}
  $a_{1l}(x,y)=ax+by+c,$ \\
  $a_{1r}(x,y)=a_1x+b_1y+c_1,$
\end{tabular}\quad
\begin{vmatrix}
  a&b \\
  a_1&b_1
\end{vmatrix}\ne 0,
\end{equation}
then we make $a_{12}(x,y)=x+a$ by elementary
transformations of columns except for the
first column. The rows of $A(\alpha,\beta)$
are linearly dependent if and only if
$(\alpha,\beta)$ is a solution of the system
\begin{equation}  \label{15.1}
\begin{vmatrix}
  a_{11}(x,y)&a_{1j}(x,y) \\
  a_{21}(x,y)&a_{2j}(x,y)
\end{vmatrix}=\begin{vmatrix}
  a_{11}&a_{1j}(x,y) \\
  x&a_{2j}(x,y)
\end{vmatrix}= 0,
\quad j=2,\dots,m.
\end{equation}
The first equation has the form
\begin{equation}  \label{15.2}
\begin{vmatrix}
  a_{11}&x+a \\
  x&bx+cy+d
  \end{vmatrix}= 0.
\end{equation}

Let $a_{11}c\ne 0$. We present \eqref{15.2}
in the form $y=a_1x^2+b_1x+c_1$, substitute
it into the other equations of the system
\eqref{15.1}, and obtain a system of
polynomial equations in $x$ of degree at
most 3. This system has at most three
solutions, and so $s\le 3$.

Let $a_{11}c= 0$. Since \eqref{15.2} is a
quadratic equation in x, $x=\alpha_1$ or
$x=\alpha_2$ for certain
$\alpha_1,\alpha_2\in k$. Substituting
$x=\alpha_i$ into the other equations of the
system \eqref{15.1} gives a system of linear
equations with respect to $y$, which has at
most one solution, and so $s\le 2$.

Hence, \eqref{15.0} does not hold for all
$l,r>1$. If there exists $j>1$ such that
$a_{1j}(x,y)=bx+a,\ b\ne 0$, then we make
$b=1$ and reason as in the previous case.
The case $a_{1j}(x,y)=a_j\in k$ for all
$j>1$ is trivial. Let us consider the
remaining case $a_{1j}(x,y)=ax+by+c,\ b\ne
0$, for a certain $j>1$. We make $$
A(x,y)=\begin{bmatrix}
  a_{11}&y&0&\dots&0 \\
  x&a_{22}(x,y)&a_{23}(x,y)&\dots
  &a_{2n}(x,y)
\end{bmatrix}
$$ by the substitution
$y_{\text{new}}=ax+by+c$ and by elementary
transformations of columns starting with the
second. If $a_{11}=0$, then the rows of
$A(\alpha,0)$ are linearly dependent for all
$\alpha\in k$. Hence $a_{11}=1$.

If the system $$ a_{2j}(x,y)=0,\quad
j=3,\dots,n, $$ has at most one solution,
then $s\le 1$. So this system is equivalent
to one equation of the form $y=ax+b$ or
$x=a$. Substituting it into
$$
\begin{vmatrix}
  a_{11}&y\\
  x&a_{22}(x,y)
\end{vmatrix}=0,
$$
we obtain a quadratic equation with respect
to $x$ or $y$. Hence $s\le 2$, a
contradiction.
\end{proof}

Let a linear matrix problem of tame type be
given by a pair $(\varGamma,\cal M)$ and let
$M\in {\cal
M}_{\underline{n}\times\underline{n}}$. We
sequentially reduce $M$ to the canonical
parametric form. If a block is reduced to a
Weyr matrix, we replace its diagonal entries
by parameters; but as soon as it becomes
clear from the form of subsequent boxes in
the process of reduction that a parameter
may possess only a finite number of values,
we replace it by these values.

The matrix that is obtained after reduction
of the first $r$ boxes will be called an
$r$-{\it matrix}; its partition into strips
(which refines the
${\underline{n}\times\underline{n}}$
partition) will be called the $r$-{\it
partition}, its strips and blocks will be
called $r$-{\it strips} and $r$-{\it
blocks}. Two $r$-matrices are {\it
equivalent} if their reduced boxes coincide.

Let $M$ be an $r$-matrix. Denote by $\bar M$
the matrix obtained from it by replacement
of all unreduced free entries with zeros.
Since the matrix problem is of tame type,
$\bar M$ is canonical for all values of
parameters, and it is reduced by
simultaneous permutations of horizontal and
vertical $r$-strips to the form
\begin{equation}   \label{14}
\bar{M}^{\vee}=
N_1(\lambda_1I)\oplus\dots\oplus
N_p(\lambda_pI) \oplus (R_1\otimes
I)\oplus\dots\oplus (R_q\otimes I),
\end{equation}
where $N_i(\lambda_iI)$ and $R_j\otimes I$
are indecomposable canonical one- and
zero-parameter canonical matrices
($R_j\otimes I$ is obtained from $R_j$ by
replacement of all its entries $a$ with
$aI$).

By the same permutation of $r$-strips, we
reduce $M$ to $M^{\vee}$ and break up it
into $(p+q)\times (p+q)$ strips conformally
to \eqref{14}. The obtained strips and
blocks will be called the {\it big strips}
and {\it big blocks} of $M^{\vee}$. (In the
terminology of \cite{ser}, the $r$-strips of
$M$ that are contained in the same big strip
are {\it linked}.)

Define the {\it weight}
$$t_M=3^{w(M)}$$
of an $r$-matrix $M$, where ${w(M)}$ is the
number of entries in all free boxes $M_i$,
$i\le r$, with the following property: $M_i$
disposes in the same big strip with a free
box $M_L$, $L<i$, containing a parameter
(that is, $M_i$ is linked with a box having
a parameter and reduces after it). Denote by
$s(M)$ the number of free entries in the
first unreduced $r$-block of $M$.

We say that an $(r+1)$-canonical matrix $B$
is an {\it extension} of an $r$-canonical
matrix $M$ and write $B\supset M$ if the
boxes $B_{1}, B_{2},\dots, B_r$ coincide
with the boxes $M_{1}, M_{2},\dots, M_r$ or
are obtained from them by replacement of
some of their parameters by scalars.

The proof of Theorem \ref{t1} bases on the
following lemma.

\begin{lemma} \label{l2}
Let $M$ be an $r$-matrix having unreduced
entries. Then the number of its
nonequivalent extensions $B\supset M$ taken
$t_B/t_M$ times is at most $4^{s(M)}$:
\begin{equation} \label{16}
\sum_{\text{nonequiv.\,}B\supset M} t_B/t_M
\le 4^{s(M)}.
\end{equation}
\end{lemma}

\begin{proof}
Let $M_{r+1}$ be the first unreduced
$r$-block of $M$ and let $M^{\vee}_{xy}$ be
the big block containing $M_{r+1}$. The
following three cases are possible.
\medskip

 \noindent {\it Case 1:
$x>p$ and $y>p$} (see \eqref{14}). Then the
horizontal and the vertical big strips of
$M^{\vee}_{xy}$ do not contain parameters,
and $t_B=t_M$ for all $B\supset M$.

(i) Let there exist a nonzero addition to
$M_{r+1}$. We make $M_{r+1}=0$, then all
$B\supset M$ are equivalent and the
inequality \eqref{16} takes the form $1\le
4^{s(M)}$.

(ii) Let there exist no nonzero addition to
$M_{r+1}$ and $M_{r+1}$ is reduced by
elementary transformations. Then each
$B\supset M$ has $B_{r+1}$ of the form
\eqref{5.2}, the number of such $z_1\times
z_2$ matrices $B_{r+1}$ is
$\min\{z_1,z_2\}+1$. The inequality
\eqref{16} takes the form $\min\{z_1,z_2\}+1
\le 4^{z_1z_2}$.

(iii) Let there exist no nonzero addition to
$M_{r+1}$ and $M_{r+1}$ is reduced by
similarity transformations. Then the box
$B_{r+1}$ of each $B\supset M$ is a
parametric Weyr matrix. The number of
parametric $z\times z$ Weyr matrices is
bounded by $3^{z-1}$ since the structure of
a matrix $W$ of the form \eqref{5.3} is
determined by the sequence
$(n_2,\dots,n_z)\in\{1,2,3\}^{z-1}$, where
$n_l=1$ if the $(l,l)$ entry of $W$ is the
first entry of $W_{\alpha_i}$, $n_l=2$ if
the $(l,l)$ entry is not the first entry of
$W_{\alpha_i}$ but the first entry of
$\alpha_iI_{m_ij}$ (see \eqref{5.4}), and
$n_l=3$ if the $(l,l)$ entry is not the
first entry of $\alpha_iI_{m_ij}$. Hence,
the number of nonequivalent extensions $B$
of $M$ is bounded by $3^{z-1}$. This proves
\eqref{16} since $t_B=t_M$ and $s(M)=z^2$.
\medskip

 \noindent {\it Case 2:
$x\le p<y$ or $y\le p<x$}. Then a horizontal
or vertical big strip of $M^{\vee}_{xy}$
contains a parameter $\lambda_l,\
l\in\{1,\dots,p\}$.

Let the parameters of $M$ take on values
from the domain of parameters. There exists
no nonzero addition to $M_{r+1}$ if and only
if
\begin{equation} \label{21}
M'=SMS^{-1}
\end{equation}
implies $M'_{r+1}=M_{r+1}$ for all
$r$-matrices $M'$ that are equivalent to $M$
and all $S\in \varGamma_{\underline{n}
\times\underline{n}}$ whose main diagonal
with respect to $r$-partition consists of
the identity $r$-blocks.\footnote{In
\cite[Theorem 1.4(b)]{ser}, the condition
``but $M_q'\ne M_q$'' must be replaced with
``and $M_q'=0$''.}

Let us partition $S$ and $M$ into
$r$-blocks:
$S=[S_{\alpha\beta}]_{\alpha,\beta=1}^e$ and
$M=[M_{\alpha\beta}]_{\alpha,\beta=1}^e$.
Since $M_{r+1}$ is an $r$-block,
$M_{r+1}=M_{\zeta\eta}$ for certain $\zeta$
and $\eta$. Presenting \eqref{21} in the
form $M'S=SM$ and equating the $(\zeta
,\eta)$ $r$-blocks, we obtain
\begin{equation}       \label{21'}
M_{\zeta 1}'S_{1\eta}+\dots+ M_{\zeta
,\eta-1}'S_{\eta-1,\eta}+ M_{\zeta \eta}' =
M_{\zeta\eta}+ S_{\zeta,
\zeta+1}M_{\zeta+1,\eta}+\dots +S_{\zeta
e}M_{e\eta}
\end{equation}
since $S$ is upper triangular with identity
diagonal $r$-blocks.

The blocks $M_{\zeta 1}',\dots, M_{\zeta
,\eta-1}'$ precede $M_{\zeta \eta}'$ so they
have been reduced and $M_{\zeta 1}'=M_{\zeta
1}, \dots,M_{\zeta ,\eta-1}'=M_{\zeta
,\eta-1}$. Moreover, each of them is nonzero
only when it is contained in the big block
$M_{xx}^{\vee}$ (they are contained in the
$x$ big horizontal strip of $M^{\vee}$ since
$M_{\zeta \eta}$ is contained in
$M_{xy}^{\vee}$, but $M^{\vee}$ is
big-block-diagonal, see \eqref{14}).
Analogously, each of $M_{\zeta+1,\eta},\dots
,M_{e\eta}$ is nonzero only when it is
contained in $M_{yy}^{\vee}$. Hence, each
$r$-block $S_{\alpha\beta}$ in \eqref{21'}
may have a nonzero factor only when it is
contained in $S_{xy}^{\vee}$. This factor
has the form $(a\lambda_l+b)I$, $a,b\in k$,
since all reduced free $r$-blocks from
$M_{xx}^{\vee}$ and $M_{yy}^{\vee}$ are zero
matrices, scalar matrices, and $\lambda_lI$.

Therefore, there exists no nonzero addition
to $M_{r+1}$ for $\lambda_l=a\in k$ if and
only if the following property holds for
each $S\in \varGamma_{\underline{n}\times
\underline{n}}$ whose main diagonal with
respect to $r$-partition consists of the
identity $r$-blocks: if the transformation
\eqref{21} given by $S$ preserves all boxes
preceding $M_{r+1}$, then
\begin{equation}    \label{23}
M_{\zeta 1}S_{1\eta}+\dots+ M_{\zeta
,\eta-1}S_{\eta-1,\eta}-
 S_{\zeta, \zeta+1}
 M_{\zeta+1,\eta}-\dots
 -S_{\zeta e}M_{e\eta}=0.
\end{equation}
The equality \eqref{23} is a linear
combination of $r$-blocks from
$S_{xy}^{\vee}$; its coefficients are linear
polynomials in $\lambda_l$.

The conditions on $r$-blocks of
$S_{xy}^{\vee}$ that ensure the preservation
of all boxes preceding $M_{r+1}$ can be
formulated in the form of a system of linear
homogeneous equations with respect to
$r$-blocks of $S$ that consists of:

(a) Linear equations with coefficients from
$k$ that give the algebra
$\varGamma_{\underline{n}
\times\underline{n}}$ as a vector space. We
restrict ourselves to those equations that
contain $r$-blocks from $S_{xy}^{\vee}$,
then they do not contain $r$-blocks outside
$S_{xy}^{\vee}$ (see \cite[p. 87]{ser}).

(b) Linear equations with coefficients from
$k$ that ensure the preservation of those
free $r$-blocks $M_{\alpha\beta}$ that are
contained in the intersection of
$M_{xy}^{\vee}$ with the boxes
$M_1,\dots,M_L$, where $M_L$ is the free box
containing the parameter $\lambda_l$. These
equations have the form \eqref{23} with the
indices $(\alpha,\beta)$ instead of $(\zeta
,\eta)$.

(c) Linear equations, whose coefficients are
linear polynomials in $\lambda_l$, that
ensure the preservation of free $r$-blocks
$M_{\alpha\beta}$ contained in the
intersection of $M_{xy}^{\vee}$ with the
boxes $M_{L+1},\dots,M_r$; the number of
entries in the boxes $M_{\alpha\beta}$ will
be denoted by $h$. They also have the form
\eqref{23} with $(\alpha,\beta)$ instead of
$(\zeta ,\eta)$.

Solving the system (a)$\cup$(b), we choose
$r$-blocks $S_1,\dots,S_n$ from
$S_{xy}^{\vee}$ such that they are arbitrary
and the other $r$-blocks from
$S_{xy}^{\vee}$ are their linear
combinations. Substituting the solution into
the system (c) and the equation \eqref{23},
we obtain a system of the form
\begin{equation}  \label{25}
\begin{matrix}
\qquad\ \ a_{11}(\lambda_l)S_1+
\dots+a_{1n}(\lambda_l)S_n=0
\\ \hdotsfor{1}\\
a_{m-1,1}(\lambda_l)S_1+
\dots+a_{m-1,n}(\lambda_l)S_n=0
\end{matrix}
\end{equation}
and, respectively, an equation
\begin{equation}  \label{25'}
a_{m1}(\lambda_l)S_1+
\dots+a_{mn}(\lambda_l)S_n=0,
\end{equation}
where $a_{ij}(\lambda_l)$ are linear
polynomials in $\lambda_l$. We take the
equations \eqref{25}--\eqref{25'} such that
the $m\times n$ matrix $A(\lambda_l)=
[a_{ij}(\lambda_l)]$ has linearly
independent rows for almost all values of
$\lambda_l$; it is possible by \cite[Sect.
3.3.2]{ser} since the matrix problem is of
tame type. Then $m\le n$.

Let there exist no nonzero addition to
$M_{r+1}$ for $\lambda_l=\alpha\in k$. Then
the equation \eqref{25'} follows from the
system \eqref{25}. Therefore, all
determinants formed by columns of the matrix
$A(\lambda_l)$ become zero for
$\lambda_l=\alpha$. These determinants are
polynomials in $\lambda_l$ of degree at most
$m$. If all the polynomials are identically
equal to 0, then the rows of $A(\lambda_l)$
are linearly dependent for all values of
$\lambda_l$ and the problem is of wild type.
Therefore, they have at most $m$ common
roots, and hence there are at most $m$
values $\alpha\in k$ of $\lambda_l$ for
which we cannot make $M_{r+1}=0$.

Let $\lambda_l$ be equal to one of these
values. The matrix $M_{r+1}$ is transformed
by equivalence transformations since
$M_{r+1}$ is not contained in a diagonal big
block. Hence each extension $B\supset M$ has
$B_{r+1}$ in the form \eqref{5.2}; the
number of nonequivalent extensions $B$ with
nonzero $B_{r+1}$ and the same value of
$\lambda_l$ is $\min\{z_1,z_2\}$, where
$z_1\times z_2$ is the size of $M_{r+1}$;
their weight $t_B\le t_M/3^{m-1}$ (since
$\lambda_l$ no longer is a parameter and
$m-1\le h$, where $h$ is defined in
paragraph (c)).

There is also one (up to equivalence)
extension $B\supset M$ with $B_{r+1}=0$ and
the parameter $\lambda_l$. Its weight
$t_B=t_M\cdot 3^{z_1z_2}$.

We have $$ \sum_{\text{nonequiv.\,}B\supset
M} t_B/t_M \le
3^{z_1z_2}+m\cdot\min\{z_1,z_2\}\cdot
3^{-m+1}\le 4^{z_1z_2}=4^{s(M)} $$ since
$m\cdot 3^{-m+1}\le 1$ and
$3^{z_1z_2}+\min\{z_1,z_2\}\le 4^{z_1z_2}$
for all natural numbers $m$, $z_1$ and
$z_2$. This proves \eqref{16}.
\medskip

 \noindent {\it Case 3:
$x\le p$ and $y\le p$}. Then the horizontal
and vertical big strips of $M^{\vee}_{xy}$
contain parameters $\lambda_l$ and
$\lambda_r$ from free boxes $M_L$ and $M_R$,
respectively. We will assume $L\le R$.

Let $l=r$. Then $M_L=M_R$ is a Weyr matrix,
$\lambda_l=\lambda_r$ is the parameter of
its block \eqref{5.4}, and $x=y$. This case
is similar to Case 2, but the matrix
$M_{r+1}$ is reduced by similarity
transformations since $M_{r+1}$ is contained
in the diagonal big block $M^{\vee}_{xx}$.
In each extension $B\supset M$, the box
$B_{r+1}$ is a Weyr matrix. The number of
parametric $z\times z$ Weyr matrices is
bounded by $3^{z-1}$ (see Case 1(iii)), so
we have $$
 \sum_{\text{nonequiv.\,}B\supset
M} t_B/t_M \le 3^{z^2}+m\cdot 3^{z-1}\cdot
3^{-m+1}\le 4^{z^2}=4^{s(M)} $$ since
$m\cdot 3^{-m+1}\le 1$ and
$3^{z^2}+3^{z-1}\le 4^{z^2}$ for all natural
numbers $m$ and $z$.

Let $l\ne r$. Then $x\ne y$; in distinction
to Case 2, the system (c) consists of linear
equations whose coefficients are linear
polynomials in $\lambda_l$ and $\lambda_r$.
Correspondingly, the system \eqref{25} and
the equation \eqref{25'} take the form
\begin{equation}  \label{25a}
\begin{matrix}
\qquad\ \ a_{11}(\lambda_l,\lambda_r)S_1+
\dots+a_{1n}(\lambda_l,\lambda_r)S_n=0
\\ \hdotsfor{1}\\
a_{m-1,1}(\lambda_l,\lambda_r)S_1+
\dots+a_{m-1,n}(\lambda_l,\lambda_r)S_n=0
\end{matrix}
\end{equation}
and
\begin{equation}  \label{25'a}
a_{m1}(\lambda_l,\lambda_r)S_1+
\dots+a_{mn}(\lambda_l,\lambda_r)S_n=0,
\end{equation}
respectively, where
$a_{ij}(\lambda_l,\lambda_r)$ are linear
polynomials in $\lambda_l$ and $\lambda_r$.

Let there exist no nonzero addition to
$M_{r+1}$ for $(\lambda_l,\lambda_r)
=(\alpha,\beta)\in k^2$. Then the equation
\eqref{25'a} follows from the system
\eqref{25a} and hence the matrix
$A(\alpha,\beta)$ (see \eqref{14.0}) has
linearly dependent rows. The set of values
of $(\lambda_l,\lambda_r)$ for which the
rows of $A(\lambda_l,\lambda_r)$ are
linearly dependent is finite (otherwise the
matrix problem is of wild type, see
\cite[Sect. 3.3.1]{ser}); assume that this
set consists of pairs
$
(\alpha_1,\beta_1),\
(\alpha_2,\beta_2),\dots,
(\alpha_s,\beta_s)\in k^2.
$

By analogy with Case 2, there are at most
$s\cdot\min\{z_1,z_2\}$ nonequivalent
extensions $B\supset M$ with nonzero
$B_{r+1}$ of size $z_1\times z_2$, their
weight $t_B\le t_M/3^{m-1}$ (since
$\lambda_l$ and $\lambda_r$ no longer are
parameters). There is also one extension
$B\supset M$ with $B_{r+1}=0$ and the
parameters $\lambda_l$ and $\lambda_r$; its
weight $t_B=t_M\cdot 3^{z_1z_2}$. We have $$
\sum_{\text{nonequiv.\,}B\supset M} t_B/t_M
\le 3^{z_1z_2}+s\cdot\min\{z_1,z_2\}\cdot
3^{-m+1}\le 4^{z_1z_2}=4^{s(M)} $$ since
$s\cdot 3^{-m+1}\le 1$ by Lemma \ref{l1} and
$3^{z_1z_2}+\min\{z_1,z_2\}\le 4^{z_1z_2}$
for all natural numbers $m$, $z_1$ and
$z_2$. This proves \eqref{16}.
\end{proof}

\begin{proof}[Proof of Theorem \ref{t1}]
Let $M$ be an $r$-matrix of size
$\underline{n} \times\underline{n}$. We will
write $M\Subset C$ if $C$ is a canonical
parametric matrix whose boxes $C_{1},
C_{2},\dots, C_r$ coincide with the boxes
$M_{1}, M_{2},\dots, M_r$ or are obtained
from them by replacement of some of their
parameters by scalars. We may add
sequentially the boxes of $C$ to the boxes
of $M$ and obtain a sequence of extensions
\begin{equation}\label{26}
  M\subset B_1 \subset B_2 \subset
  \dots \subset B_{l-1}\subset B_l=C,
\end{equation}
where $B_i$ is an $(r+i)$-matrix and $l+r$
is the number of boxes of $C$. The length
$l$ of this sequence may be changed if we
change $C$; the greatest length $l$ will be
called the {\it dept} of $M$ and will be
denoted by $l(M)$.

We prove by induction in $l(M)$ that
\begin{equation} \label{27}
\sum_{C\Supset M} t_C/t_M \le
4^{\bar{s}(M)},
\end{equation}
where $\bar{s}(M)$ is the number of
unreduced free entries in $M$.

If $l(M)=1$, this inequality follows from
Lemma \ref{l2}. Let $l(M)\ge 2$ and
\eqref{27} holds for all $r'$-matrices whose
dept is less than $l(M)$. Then
\begin{alignat*}{2}
\sum_{C\Supset M} t_C/t_M
&=\sum_{\text{nonequiv.\,}B\supset M}
\sum_{C\Supset B} t_C/t_B \cdot t_B/t_M &&\\
&= \sum_{\text{nonequiv.\,}B\supset M}
t_B/t_M \sum_{C\Supset B} t_C/t_B &&\\
&\le \sum_{\text{nonequiv.\,}B\supset M}
t_B/t_M \cdot 4^{\bar{s}(B)} &&
 \quad \text{by the induction
 hypothesis}\\
&=4^{\bar{s}(M)-s(M)}
\sum_{\text{nonequiv.\,}B\supset M} t_B/t_M
&&\\
&=4^{\bar{s}(M)-s(M)}\cdot 4^{s(M)} &&
 \quad \text{by Lemma \ref{l2}}\\
&=4^{\bar{s}(M)}; &&
 \end{alignat*}
that proves \eqref{27}. The substitution of
the 0-canonical matrix $0$ for $M$ in
\eqref{27} gives
$$ \sum_{C\Supset 0} t_C \le
4^{s(\underline{n})}.
$$
This proves Theorem \ref{t1} since the sum
is taken over all canonical parametric
matrices and $t_C\ge 1$ by the definition of
weight.
\end{proof}

Now we extend Theorem \ref{t1} to matrix
problems, in which row- and
column-transformations are separated.

Let $\varGamma\subset k^{t\times t}$ and
$\Delta\subset k^{l\times l}$ be two basic
matrix algebras and let ${\cal N} \subset
k^{t\times l}$ be a vector space such that
$$
\varGamma{\cal N} \subset {\cal N}\quad
\text{and} \quad {\cal N} \Delta \subset
{\cal N}.
$$
By a {\it separated matrix problem given by}
$(\varGamma, \Delta,{\cal N})$, we mean the
canonical form problem for matrices $N\in
{\cal N}_{\underline{m}\times\underline{n}}$
in which the row transformations are given
by $\varGamma$ and the column
transformations are given by $\Delta $:
$$
N\longmapsto CNS,\quad C\in
\varGamma_{\underline{m}\times
\underline{m}}^*,\ S\in \Delta
_{\underline{n}\times\underline{n}}^*.
$$
Following \cite[Lemma 2.3]{ser}, we may
consider this matrix problem as the linear
matrix problem given by the pair
$(\varGamma\times \Delta ,\ 0\diagdown {\cal
N})$ (see \eqref{3.4aa}), where $0\diagdown
{\cal N}$ denotes the vector space of
${(t+l)\times (t+l)}$ matrices of the form
$$
\begin{bmatrix}
  0 & X\\ 0&0
\end{bmatrix},\qquad X\in{\cal N}.
$$
This permits to extend Theorem \ref{t1} to
separated matrix problems.

\begin{theorem}   \label{t1'}
If a separated matrix problem is of tame
type, then the number of its canonical
parametric matrices of size
$\underline{m}\times\underline{n}$ is
bounded by $4^{s(\underline{m},
\underline{n})}$, where $s(\underline{m},
\underline{n})$ is the number of free
entries in an
$\underline{m}\times\underline{n}$ matrix.
\end{theorem}

\section{Number of modules}
\label{s4}

The problem of classifying modules over
finite dimensional algebra $A$ reduces to a
linear matrix problem; its canonical
matrices determine a full system of
nonisomorphic modules over $A$ (see
\cite[Sect. 2.5]{ser}), which will be called
{\it canonical}. If $A$ is of tame type,
then the set of canonical right modules of a
fixed dimension partitions into a finite
number of series that are determined by
canonical parametric matrices of the form
\eqref{12}. In this section, we prove the
following estimate.

\begin{theorem} \label{t2}
If $A$ is an algebra of tame type and
$f(d,A)$ is the number of series of
canonical right $A$-modules of dimension at
most $d$, then
\begin{equation}\label{3.1}
f(d,A)\le {\binom {d+r} r}
4^{d^2(\delta_1^2+\dots +\delta_r^2)}\le
(d+1)^r4^{d^2(\dim A)^2},
\end{equation}
where $r$ is the number of nonisomorphic
indecomposable projective left $A$-modules,
and $\delta_1,\dots,\delta_r$ are their
dimensions.
\end{theorem}

Without loss of generality, we will prove
Theorem \ref{t2} for basic matrix algebras
(see \eqref{3.00}). Indeed, $A$ is
isomorphic to the subalgebra $B\subset
\End_k A$ consisting of all linear operators
\begin{equation}\label{3.0}
  \hat{a}: x\mapsto ax,\qquad a\in A,
\end{equation}
on the space $_kA$. There exists a basis of
$_kA$ in which the matrices of $B$ form an
algebra $\varGamma_{\underline{n}\times
\underline{n}}$, where $\varGamma\subset
k^{t\times t}$ is a basic matrix algebra and
$\underline{n}=(n_1,\dots,n_t)\in {\mathbb
N}^t$, see \cite[Theorem 1.1]{ser}. By the
Morita theorem \cite{dr_ki}, the categories
of representations of
$\varGamma_{\underline{n}\times
\underline{n}}$ and its basic algebra
$\varGamma$ are equivalent, hence
$$
f(d,A)=f(d,\varGamma_{\underline{n}\times
\underline{n}})= f(d,\varGamma).
$$
Furthermore, the replacement of
$\varGamma_{\underline{n}\times
\underline{n}}$ with $\varGamma$ preserves
the number $r$ of nonisomorphic
indecomposable projective left modules and
reduces their dimensions.

The algebra $\varGamma$ determines the
equivalence relation \eqref{4.0a} in the set
of indices $T=\{1,\dots,t\}$. Let ${\cal
I}_1,\dots, {\cal I}_r$ be the equivalence
classes, put
\begin{equation}\label{3.1'}
 e_{\alpha}=\sum_{i\in {\cal
I}_{\alpha}} e_{ii},
\end{equation}
where $e_{ij}$ are the matrix units of
$k^{t\times t}$. Define the matrix
\begin{equation}\label{3.1''}
L=[l_{\alpha\beta}]_{\alpha,\beta=1}^r,
\qquad l_{\alpha\beta}=\dim
e_{\alpha}Re_{\beta},
\end{equation}
where $R=\rad\varGamma$ is the radical of
$\varGamma$ consisting of all its matrices
with zero diagonal.

\begin{lemma} \label{l3}
If $\varGamma\in k^{t\times t}$ is a basic
matrix algebra of tame type, then
\begin{equation}\label{3.1a}
  f(d,\varGamma)\le
  \sum_{q_1+\dots+q_r\le d}
  4^{[q_1,\dots,q_r]L\cdot
  ([q_1,\dots,q_r]L)^T},
\end{equation}
where $q_1,\dots,q_r$ are nonnegative
integers.
\end{lemma}

Let us show that \eqref{3.1a} implies
Theorem \ref{t2}. By \eqref{3.1'},
$$ I=e_1+\dots+e_r $$
is a decomposition of the identity of
$\varGamma$ into a sum of minimal orthogonal
idempotents, and so $\varGamma e_1,\dots,
\varGamma e_r$ are all nonisomorphic
indecomposable projective left modules over
$\varGamma$. The number of summands in
\eqref{3.1a} is equal to the number of
solutions of the inequality
\begin{equation}\label{3.1aaa}
x_1+\dots+x_r\le d
\end{equation}
in nonnegative integers; it equals ${\binom
{d+r} r}$ by \cite[Sect. 1.2]{sta}.
  Since $q_{\alpha} \le
d$, $[q_1,\dots,q_r]L\cdot
([q_1,\dots,q_r]L)^T\le d^2[1,\dots,1]L\cdot
([1,\dots,1]L)^T= d^2(\delta_1^2+\dots+
\delta_r^2)$, where $\delta_{\beta} =
  [1,\dots,1]\cdot [l_{1\beta},
  \dots,l_{r\beta}]^T =
  l_{1\beta}+\dots+ l_{r\beta}=
  \dim e_{1}Re_{\beta}+\dots+
  \dim e_{r}Re_{\beta}=
  \dim (e_{1}+\dots+ e_{r})Re_{\beta}=
  \dim Re_{\beta}=
  \dim \varGamma e_{\beta}-1$.
This proves the first inequality in
\eqref{3.1}. We have
$${\binom {d+r} r}\le
(d+1)^2$$ since each $x_i$ in \eqref{3.1aaa}
possesses at most $d+1$ values
$0,1,\dots,d$. We also have
$\delta_1^2+\dots +\delta_r^2\le
(\delta_1+\dots +\delta_r)^2= (\dim
\varGamma e_1+\dots+\dim \varGamma e_r)^2=
(\dim \varGamma (e_1+\dots+ e_r))^2=(\dim
\varGamma)^2\le (\dim A)^2$. This proves the
second inequality in \eqref{3.1}.

\begin{proof}[Proof of Lemma \ref{l3}]
{\it Step 1: reduction to a matrix problem.}
The reduction to a linear matrix problem
given in \cite{ser} is a light modification
of Drozd's reduction \cite{dro1} (see also
\cite{dro2} and \cite{cra}). It bases on the
construction, for every right module $M$
over $\varGamma$, an exact sequence
\begin{gather}
P\stackrel{\varphi}{\longrightarrow} Q
\stackrel{\psi}{\longrightarrow}
M\longrightarrow 0, \label{3.2}\\
\Ker\varphi\subset \rad P, \quad
\im\varphi\subset \rad Q, \label{3.2a}
\end{gather}
where $P$ and $Q$ are projective right
modules. The homomorphism $\varphi$ is
defined by $P$, $Q$, and $M$ up to
transformations
\begin{equation}\label{3.3}
  \varphi\longmapsto g\varphi f,
  \qquad f\in \Aut_{\varGamma} P,\quad g\in
\Aut_{\varGamma} Q.
\end{equation}

Let us show briefly (details in \cite{ser})
that the problem of classifying $\varphi$ up
to these transformations reduces to a
separated matrix problem given by the triple
$(\varGamma,\varGamma,\rad\varGamma)$.

Decompose $P$ and $Q$ from \eqref{3.2} into
direct sums of indecomposable projective
modules:
\begin{equation}\label{3.3'}
 P=(e_1\varGamma)^{p_1}\oplus\dots\oplus
(e_r\varGamma)^{p_r},\quad
Q=(e_1\varGamma)^{q_1}\oplus\dots\oplus
(e_r\varGamma)^{q_r},
\end{equation}
where $X^{l}:=X\oplus\dots\oplus X$ ($l$
times) and $e_i$ are defined by
\eqref{3.1'}. Then the homomorphism
$\varphi$ becomes the $q\times p=
(q_1+\dots+q_r)\times (p_1+\dots+p_r)$
matrix $\varphi=
[\varphi_{xy}]_{x=1,}^q{}_{y=1}^p$, which we
partition into $r$ horizontal and $r$
vertical strips of sizes $q_1,\dots, q_r$
and $p_1,\dots, p_r$. Denote by
$$
\alpha=\alpha(x)\quad\text{and}\quad
\beta=\beta(y)
$$
the indices of the vertical and the
horizontal strips containing $\varphi_{xy}$.
Then $\varphi_{xy}: e_{\beta}\varGamma \to
e_{\alpha}\varGamma$ and is determined by
$\varphi_{xy} (e_{\beta})=
e_{\alpha}\varphi_{xy} (e_{\beta})\in
e_{\alpha}\varGamma$. Since $\varphi_{xy}$
is a homomorphism and $e_{\beta}$ is an
idempotent, $\varphi_{xy}
(e_{\beta})=\varphi_{xy}
(e_{\beta}^2)=\varphi_{xy}
(e_{\beta})e_{\beta}$. Hence, $\varphi_{xy}
(e_{\beta})=e_{\alpha}\varphi_{xy}
(e_{\beta}) e_{\beta}\in e_{\alpha}\varGamma
e_{\beta}$. By \eqref{3.2a}, $$
\im\varphi\subset \rad Q
=(e_1R)^{q_1}\oplus\dots\oplus (e_rR)^{q_r},
$$ where $R=\rad\varGamma$. We have
$\varphi_{xy} (e_{\beta(y)})\in
e_{\alpha(x)}R e_{\beta(y)}$.

If a matrix $a=[a_{ij}]_{i,j=1}^t\in
\varGamma$ belongs to $e_{\alpha}R
e_{\beta}$, then it is determined by its
submatrix $\bar{a}=[a_{ij}]_{(i,j)\in {\cal
I}_{\alpha}\times{\cal I}_{\beta}}$ since
all entries outside of $\bar{a}$ are zero by
\eqref{3.1'}. The size of $\bar{a}$ is
$h(\alpha)\times h(\beta)$, where $
h(\alpha)$ is the number of elements in
${\cal I}_{\alpha}$. Therefore, the
homomorphism $\varphi=
[\varphi_{xy}]_{x=1,}^q{}_{y=1}^p$ is
determined by the block matrix
\begin{equation}\label{3.4}
[\overline{\varphi_{xy}
(e_{\beta(y)})}]_{x=1,}^q{}_{y=1}^p
\end{equation}
of size
$$
(q_1h(1)+\dots +q_rh(r))\times
(p_1h(1)+\dots +p_rh(r)).
$$
Permuting rows and columns of this matrix to
order them in accordance with their position
in $\varGamma$, we obtain a block matrix
${\Phi}\in R_{\underline{m}\times
\underline{n}}$, where $m_i:=q_{\alpha}$ if
$i\in{\cal I}_{\alpha}$ and $n_j:=p_{\beta}$
if $j\in{\cal I}_{\beta}$. In the same way,
the automorphisms $f\in \Aut_{\varGamma} P$
and $g\in \Aut_{\varGamma} Q$ are determined
by nonsingular matrices from
$\varGamma_{\underline{m}\times
\underline{m}}$ and
$\varGamma_{\underline{n}\times
\underline{n}}$.

Hence, the problem of classifying modules
over $\varGamma$ reduces to the canonical
form problem for matrices ${\Phi}\in
R_{\underline{m}\times \underline{n}}$ up to
transformations
\begin{equation}\label{3.5}
  \Phi\longmapsto F\Phi G,
  \qquad F\in \varGamma_{\underline{m}\times
\underline{m}}^{*},\quad G\in
\varGamma_{\underline{n}\times
\underline{n}}^{*}.
\end{equation}
Let
\begin{equation}\label{3.5aa}
H_1,\dots,H_t
\end{equation}
be the vertical strips of $\Phi$ with
respect to $\underline{m}\times
\underline{n}$ partition. The condition
$\Ker\varphi\subset \rad P$ from
\eqref{3.2a} means that
\begin{equation}       \label{3.5a}
\parbox{25em}
{there are not an equivalence class ${\cal
I}_{\alpha}=\{j_1,\dots,j_{h(\alpha)}\}$ and
a transformation \eqref{3.5} making zero the
last column in each of
$H_{j_1},\dots,H_{j_{h(\alpha)}}$
simultaneously.}
\end{equation}
\medskip

\noindent {\it Step 2: an estimate.} Let the
module $M$ in \eqref{3.2} has dimension at
most $d$. By \eqref{3.2}, \eqref{3.3'}, and
the condition $\im\varphi\subset \rad Q$
from \eqref{3.2a},
\begin{equation}\label{3.6}
 q_1+\dots+ q_r= \dim Q/ \rad Q
 \le \dim Q/\im \varphi =
\dim M\le d.
\end{equation}
Each summand
$(e_{\alpha}\varGamma)^{p_{\alpha}}$ in the
decomposition \eqref{3.3'} of $P$ determines
the equivalence class ${\cal I}_{\alpha}=
\{j_1,\dots,j_{h(\alpha)}\}$ and corresponds
to the strips
$H_{j_1},\dots,H_{j_{h(\alpha)}}$ of $\Phi$
(see \eqref{3.5aa}); these strips are
reduced by simultaneous elementary
transformations and each of them has
$p_{\alpha}$ columns.

Let us prove that
\begin{equation}\label{3.7}
 p_{\alpha}\le [q_1,\dots,q_r]\cdot
 [l_{1\alpha},\dots,l_{r\alpha}]^T,
\end{equation}
where $[l_{1\alpha},\dots,l_{r\alpha}]^T$ is
a column of the matrix \eqref{3.1''}. Put
$$
n_{\iota}=[q_1,\dots,q_r]\cdot [\dim
e_1\varGamma
e_{j_{\iota}j_{\iota}},\dots,\dim
e_r\varGamma e_{j_{\iota}j_{\iota}}]^T,
\qquad 1\le \iota\le h(\alpha),
$$
where $e_{jj}$ are matrix units. By
\eqref{3.1'}, $$[q_1,\dots,q_r]\cdot
[l_{1\alpha},\dots,l_{r\alpha}]^T= n_1+\dots
+n_{h(\alpha)}.$$

Suppose that \eqref{3.7} does not hold,
i.e.
$$
p_{\alpha}\ge n_1+\dots +n_{h(\alpha)}+1,
$$
and show that there is a transformation
making zero the $(n_1+\dots
+n_{h(\alpha)}+1)$st column in each of
$H_{j_1},\dots,H_{j_{h(\alpha)}}$
simultaneously, to the contrary with
\eqref{3.5a}. It suffices to show that there
is a transformation making zero the
$(n_1+\dots +n_{h(\alpha)}+1)$st column in
all free blocks from
$H_{j_1},\dots,H_{j_{h(\alpha)}}$ since the
other blocks are their linear combinations.

The number of rows in free blocks of
$H_{j_1}$ is equal to $ n_1$; by elementary
transformations of columns, we maximize the
rank of the first $n_1$ columns of these
blocks, and then make zero the other their
columns (by the definition of admissible
transformations, the same transformations
are produced within the strips
$H_{j_2},\dots,H_{j_{h(\alpha)}}$).  The
number of rows in free blocks of $H_{j_2}$
is $n_2$; by elementary transformations with
the $n_1+1,n_1+2,\dots$ columns, we maximize
the rank of the $n_1+1,n_1+2,\dots$
$n_1+n_2$ columns of these blocks, and then
make zero the $n_1+n_2+1,n_1+n_2+2,\dots$
columns in free blocks of $H_{j_2}$ (the
same transformations are produced within the
strips $H_{j_1}, H_{j_3} \dots,
H_{j_{h(\alpha)}}$; they do not spoil the
made zeros in $H_{j_1}$), and so on. At
last, we reduce $H_{j_{h(\alpha)}}$ and
obtain $\Phi$ in which the $(n_1+\dots+
n_{h(\alpha)}+1)$st column is zero in all
free boxes of
$H_{j_1},\dots,H_{j_{h(\alpha)}}$. This
proves \eqref{3.7}.

Therefore, each module $M$ of dimension at
most $d$ may be given by a sequence
\eqref{3.2}, in which $P$ and $Q$ are of the
form \eqref{3.3'} with $p_i$ and $q_j$
satisfying \eqref{3.6} and \eqref{3.7}. To
make
$$
p_{\alpha}= [q_1,\dots,q_r]\cdot
[l_{1\alpha},\dots,l_{r\alpha}]^T,
$$
we add, if necessary, additional summands to
the decomposition \eqref{3.3'} of $P$ and
put $\varphi$ equaling 0 on the new
summands. Correspondingly, we omit the first
condition in \eqref{3.2a} and the condition
\eqref{3.5a} on the matrix $\Phi$. The
number of free entries in $\Phi$ becomes
equal to
$$
 [q_1,\dots,q_r]L[p_1,\dots,p_r]^T=
 {[q_1,\dots,q_r]L\cdot
  ([q_1,\dots,q_r]L)^T};
$$
this proves \eqref{3.1a} in view of
\eqref{3.6} and Theorem \ref{t1'}.
\end{proof}

\end{document}